\documentclass[amstex,12pt]{article}
\usepackage{amssymb,amsmath,amsthm,graphicx}
\newtheorem{theorem}{Theorem}[section]
\newtheorem{corollary}{Corollary}[section]
\newtheorem{lemma}{Lemma}[section]
\newtheorem{definition}{Definition}[section]
\newtheorem{remark}{Remark}[section]

\newcommand{\beq}{\begin{equation}}
\newcommand{\eeq}{\end{equation}}
\newcommand{\beqn}{\begin{eqnarray}}
\newcommand{\eeqn}{\end{eqnarray}}

\allowdisplaybreaks \baselineskip 20pt
\begin{document}

\title{Almost automorphic functions on the quantum time scale and applications\thanks{This work is supported by
the National Natural Sciences Foundation of People's Republic of
China under Grant 11361072.}  }
\author{ Yongkun Li\\
Department of Mathematics, Yunnan University\\
Kunming, Yunnan 650091\\
People's Republic of China
}

\date{}
\maketitle \allowdisplaybreaks
\begin{abstract}
In this paper, we first propose two types of concepts of almost automorphic functions on the quantum time scale.
Secondly, we study some  basic properties of almost automorphic functions on the quantum time scale. Then, we introduce a transformation between functions defined on the quantum time scale and functions defined on the set of generalized integer numbers, by using this transformation we give  equivalent definitions of almost automorphic functions on the quantum time scale. Finally, as an application of our results, we establish the existence of almost automorphic solutions of linear and semilinear dynamic equations on the quantum time scale.
\end{abstract}


\allowdisplaybreaks
\section{Introduction}

\setcounter{section}{1}
\setcounter{equation}{0}
\indent

\setcounter{equation}{0}
 \indent \allowdisplaybreaks

Since the theory of quantum calculus  has important applications in quantum theory (see Kac and Cheung
\cite{q1}), it has
received much attention. For example,
since Bohner and Chieochan \cite{q5} introduced the  concept of periodicity for functions defined on the quantum time scale, quite a few authors  have devoted themselves to the study of periodicity  for dynamic equations on the quantum time scale (\cite{q2a,q2,q3,q4}).

However, in reality, almost periodic phenomenon is  more common and complicate than periodic one. In addition, the almost automorphy is a generalization of almost periodicity and plays an important role in understanding the almost periodicity.
Therefore, to study the almost automorphy of  dynamic equations on the quantum time scale
 is  more interesting and more challenge.

Our main purpose of this paper is   to propose two types of definitions of almost automorphic functions on the quantum time scale,   study some of their basic properties  and establish the existence  of almost automorphic solutions of non-autonomous linear dynamic equations on the quantum time scale.

The organization of this paper is as follows: In Section 2, we
introduce some notations and definitions of time scale calculus. In Section 3, we propose the
concepts of  almost automorphic functions on the quantum time
scale and investigate some of their basic properties.  In Section 4, we introduce a transformation and give an equivalent definition of   almost automorphic functions on the quantum time scale.  In Section 5,  as an  application of the results, we study the existence of almost automorphic solutions for semilinear dynamic equations on the quantum time scale. We draw a conclusion in Section 6.

\section{Preliminaries}

\setcounter{section}{2}
\setcounter{equation}{0}
\indent

In this section, we shall recall some basic definitions of time scale calculus.

A time scale $\mathbb{T}$ is an arbitrary nonempty closed subset of the real numbers, the forward and backward jump operators $\sigma$, $\rho:\mathbb{T}\rightarrow \mathbb{T}$ and the forward graininess $\mu:\mathbb{T}\rightarrow \mathbb{R}^{+}$ are defined, respectively, by
\[
\sigma(t):=\inf \{s\in\mathbb{T}:s> t\},\,\,\rho(t):=\sup\{s\in\mathbb{T}:s<t\}\,\,
\text{and}\,\,\mu(t)=\sigma(t)-t.
\]

A point $t$ is said to be left-dense if $t>\inf\mathbb{T}$ and $\rho(t)=t$, right-dense if $t<\sup\mathbb{T}$ and $\sigma(t)=t$, left-scattered if $\rho(t)<t$ and right-scattered if $\sigma(t)>t$. If $\mathbb{T}$ has a left-scattered maximum $m$, then $\mathbb{T}^{\kappa}=\mathbb{T}\backslash m$, otherwise $\mathbb{T}^{\kappa}=\mathbb{T}$. If $\mathbb{T}$ has a right-scattered minimum $m$, then $\mathbb{T}_{\kappa}=\mathbb{T}\backslash m$, otherwise $\mathbb{T}^{k}=\mathbb{T}$.

Let $\mathbb{X}$ be a $($real or complex$)$ Banach space.  A function $f : \mathbb{T}\rightarrow \mathbb{X}$ is right-dense continuous or rd-continuous provided it is continuous at right-dense points in $\mathbb{T}$ and its left-sided limits exist (finite) at left-dense points in $\mathbb{T}$. If $f$ is continuous at each right-dense point and each left-dense point, then $f$ is said to be a continuous function on $\mathbb{T}$.

For $f:\mathbb{T}\rightarrow\mathbb{X}$ and $t\in{\mathbb{T}^{k}}$, then $f$ is called delta differentiable at $t\in{\mathbb{T}}$ if there exists $c\in\mathbb{X}$ such that for given any $\varepsilon\geq{0}$, there is an open neighborhood $U$ of  $t$ satisfying
\[
||[f(\sigma(t))-f(s)]-c[\sigma(t)-s]||\leq\varepsilon||\sigma(t)-s|
\]
for all $s\in U$. In this case, $c$ is called the delta derivative of $f$ at $t\in{\mathbb{T}}$, and is denoted by $c=f^{\Delta}(t)$. For $\mathbb{T}=\mathbb{R}$, we have $f^{\Delta}=f^{'}$, the usual derivative, for $\mathbb{T}=\mathbb{Z}$ we have the backward difference operator, $f^{\Delta}(t)=\Delta f(t):=f(t+1)-f(t)$£¬ and for $\mathbb{T}=\overline{q^{\mathbb{Z}}} (q>1)$, the quantum time scale, we have  the $q$-derivative
\[
f^\Delta(t):=D_qf(t)=\displaystyle\left\{\begin{array}{lll}\frac{f(qt)-f(t)}{(q-1)t},&t\neq 0,\\
\lim\limits_{t\rightarrow 0}\frac{fqt)-f(t)}{(q-1)t},&t=0.
\end{array}\right.
\]
\begin{remark}Note that
\[
D_qf(0)=\frac{df(0)}{dt}
\]
if $f$ is continuously differentiable.
\end{remark}
A function $p:\mathbb{T}\rightarrow\mathbb{R}$ is called regressive provided $1+\mu(t)p(t)\neq 0$ for all $t\in{\mathbb{T}^{\kappa}}$.
An $n\times n$-matrix-valued function $A$ on a time scale
$\mathbb{T}$ is called regressive provided
$I+\mu(t)A(t)$ is invertible for all $t\in\mathbb{T}^\kappa$.

\begin{definition}\label{def21}\cite{liw} A time scale $\mathbb{T}$ is called an almost periodic time scale if
\begin{eqnarray*}
\Pi=\big\{\tau\in\mathbb{R}: t\pm\tau\in\mathbb{T}, \forall t\in{\mathbb{T}}\big\}\neq\{0\}.
\end{eqnarray*}
\end{definition}
For more details about the theory of time scale calculus and the theory of quantum calculus, the reader may want to consult \cite{q1, t1,t2,tt}.

\section{Almost automorphic functions on the quantum time scale}

\setcounter{equation}{0}
\indent

In this section, we propose two types of concepts of almost automorphic functions on the quantum time scale and study some of their basic properties. Our first type of concepts of almost automorphic functions on the quantum time scale is as follows:

\begin{definition}\label{d2.1}
Let $\mathbb{X}$ be a $($real or complex$)$ Banach space and $f:\overline{q^\mathbb{Z}}\rightarrow \mathbb{X}$ a $($strongly$)$ continuous function. We say that $f$ is almost automorphic if for every sequence of integer numbers $\{s'_{n}\}\subset \mathbb{Z}$, there exists a subsequence $\{s_{n}\}$ such that:
\[
g(t):=\lim\limits_{n\rightarrow\infty}f(tq^{s_{n}})
\]
is well defined for each $t\in \overline{q^\mathbb{Z}}$ and
\[
\lim\limits_{n\rightarrow\infty}g(tq^{-s_{n}})=f(t)
\]
for each $t\in \overline{q^\mathbb{Z}}$.
\end{definition}

\begin{remark}\label{r1}
Since $\overline{q^\mathbb{Z}}$ has only one right dense point $0$ and all of the other points of it are  isolated points. So,
$f:\overline{q^\mathbb{Z}}\rightarrow \mathbb{X}$ a $($strongly$)$ continuous function if and only if $\lim\limits_{t\rightarrow 0^+}f(t)=f(0)$.
\end{remark}

\begin{theorem}\label{t21}
If $f, f_{1}$ and $f_{2}$ are almost automorphic functions $\overline{q^\mathbb{Z}}\rightarrow \mathbb{X}$, then the following are true:
\begin{itemize}
  \item [$(i)$] $f_{1}+f_{2}$ is almost automorphic.
  \item [$(ii)$] $cf$ is almost automorphic for every scalar c.
  \item [$(iii)$] $f_{a}(t)\equiv f(tq^a)$ is almost automorphic for each fixed $a\in \mathbb{Z}.$
  \item [$(iv)$] $\sup\limits_{t\in\mathbb{R}}\|f(t)\|<\infty$, that is, $f$ is a bounded function.
  \item [$(v)$] The range $R_{f}=\{f(t)|t\in \overline{q^\mathbb{Z}}\}$ of $f$ is relatively compact in $\mathbb{X}.$
\end{itemize}
\end{theorem}
\begin{proof}
The proofs of  $(i)$, $(ii)$, and $(iii)$ are obvious.

The proof of $(iv)$. If $(iv)$ is no true, then $\sup\limits_{t\in \overline{q^\mathbb{Z}}}\|f(t)\|=\infty$. Hence,  there exists a sequence  $\{s'_{n}\}\subset \mathbb{Z}$ such that
\[
\lim\limits_{n\rightharpoonup\infty}\|f(q^{s'_{n}})\|=\infty.
\]
Since $f$ is almost automorphic, one can extract a subsequence $\{s_{n}\}\subset\{s'_{n}\}$ such that
\[
\lim\limits_{n\rightharpoonup\infty}f(q^{s_{n}})=\xi
\]
exists, that is, $\lim\limits_{n\rightharpoonup\infty}\|f(q^{s_{n}})\|=\|\xi\|<\infty$, which is a contradiction. The proof of  $(iv)$ is completed.

The proof of $(v)$. For any sequence $\{f(q^{s'_{n}})\}$ in $R_{f}$, where $\{s'_n\}\subset \mathbb{Z}$, because $f$ is almost automorphcic, so one can extract a subsequence $\{s_{n}\}$ of $\{s'_{n}\}$ such that
\[
\lim\limits_{n\rightarrow\infty}f(q^{s_{n}})=g(1).
\]
Thus, $R_{f}$ is relatively compact in $\mathbb{X}$. The proof is complete.
\end{proof}

\begin{remark}\label{p1}
It is easy to see that
\[
\sup\limits_{t\in \overline{q^\mathbb{Z}}}\|g(t)\|\leq\sup\limits_{t\in \overline{q^\mathbb{Z}}}\|f(t)\|,
\]
and $R_{g}\subseteq\bar{R}_{f}$, where $g$ is the function that appears in Definition \ref{d2.1}.
\end{remark}

\begin{theorem}\label{t2.2}
If $f:\overline{q^\mathbb{Z}}\rightarrow \mathbb{X}$ is almost automorphic, define a function $f^{\star}: \overline{q^\mathbb{Z}}\setminus\{0\}\rightarrow \mathbb{X}$  by $f^{\star}(t)\equiv f(t^{-1})$, if $f^{\star}(0):=\lim\limits_{n\rightarrow -\infty}f^{\star}(q^n)$ exists.
Then $f^{\star}: \overline{q^\mathbb{Z}}\rightarrow \mathbb{X}$  is almost automorphic.
\end{theorem}
\begin{proof}
For any given sequence $\{s'_n\}\subset \mathbb{Z}$,  there exists a subsequence $\{s_{n}\}$ of $\{s'_n\}$ such that
\[
\lim\limits_{n\rightarrow\infty}f(tq^{s_{n}})=g(t)
\]
is well defined for each $t\in \overline{q^\mathbb{Z}}$ and
\[
\lim\limits_{n\rightarrow\infty}g(tq^{-s_{n}})=f(t)
\]
for each $t\in \overline{q^\mathbb{Z}}$.

Define a function $g^{\star}(t)\equiv g(t^{-1}), t\in \overline{q^\mathbb{Z}}$ and set $\sigma_{n}=-s_{n}, n=1,2,\ldots$, we get
\begin{eqnarray*}
\lim\limits_{n\rightarrow\infty}f^{\star}(tq^{\sigma_{n}})&=&
\lim\limits_{n\rightarrow\infty}f(t^{-1}q^{-\sigma_{n}})\\
&=&\lim\limits_{n\rightarrow\infty}f(t^{-1}q^{s_{n}})\\
&=&g(t^{-1})\\
&=&g^{\star}(t)
\end{eqnarray*}
and
\begin{eqnarray*}
\lim\limits_{n\rightarrow\infty}g^{\star}(tq^{-\sigma_{n}})&=&
\lim\limits_{n\rightarrow\infty}g(t^{-1}q^{\sigma_{n}})\\
&=&\lim\limits_{n\rightarrow\infty}g(t^{-1}q^{-s_{n}})\\
&=&f(t^{-1})\\
&=&f^{\star}(t)
\end{eqnarray*}
pointwise on $\overline{q^\mathbb{Z}}$.
Since $f^{\star}(0)=\lim\limits_{n\rightarrow -\infty}f^{\star}(q^n)$ exists, $f^{\star}: \overline{q^\mathbb{Z}}\rightarrow \mathbb{X}$ is well defined and continuous.
Thus, $f^{\star}(t)$ is almost automorphic. The proof is complete.
\end{proof}

\begin{theorem}\label{t2.3}
Let $\mathbb{X}$ and $\mathbb{Y}$ be two Banach spaces and $f:\overline{q^\mathbb{Z}}\rightarrow \mathbb{X}$ an almost automorphic function. If $\phi:\mathbb{X}\rightarrow \mathbb{Y}$ is a continuous function, then the composite function $\phi(f):\overline{q^\mathbb{Z}}\rightarrow \mathbb{Y}$ is almost automorphic.
\end{theorem}
\begin{proof}Since $f$ is almost automorphic,
for any sequence $\{s'_n\}\subset \mathbb{Z}$, we can extract a subsequence $\{s_{n}\}$ of $\{s'_n\}$ such that
\[
\lim\limits_{n\rightarrow\infty}f(tq^{s_{n}})=g(t)
\]
is well defined for each $t\in \overline{q^\mathbb{Z}}$ and
\[
\lim\limits_{n\rightarrow\infty}g(tq^{-s_{n}})=f(t)
\]
for each $t\in \overline{q^\mathbb{Z}}$.

Since $\phi(f):\overline{q^\mathbb{Z}}\rightarrow \mathbb{Y}$ is continuous, we have
\[
\lim\limits_{n\rightarrow\infty}\varphi(f(tq^{s_{n}}))=\varphi\left(\lim\limits_{n\rightarrow\infty}f(tq^{s_{n}})\right)=\varphi(g(t))
\]
is well defined for each $t\in \overline{q^\mathbb{Z}}$ and
\[
\lim\limits_{n\rightarrow\infty}\varphi(g(tq^{-s_{n}}))=\varphi\left(\lim\limits_{n\rightarrow\infty}g(tq^{-s_{n}})\right)=\varphi(f(t))
\]
for each $t\in \overline{q^\mathbb{Z}}$.

That is, the composite function $\phi(f):\overline{q^\mathbb{Z}}\rightarrow \mathbb{Y}$ is almost automorphic. The proof is complete.
\end{proof}

\begin{corollary}\label{c2.1}
If $A$ is a bounded linear operator in $\mathbb{X}$ and $f:\overline{q^{\mathbb{Z}}}\rightarrow \mathbb{X}$ an almost aytomorphic function, then $A(f)(t)$ is also almost automophic.
\end{corollary}
\begin{proof}
Obvious.
\end{proof}
\begin{theorem}\label{t2.4}
 Let $f$ be almost automorphic. If $f(q^n)=0$ for all $n>n_0$ for some integer number $n_0$, then $f(t)\equiv0$ for all $t\in \overline{q^\mathbb{Z}}$.
 \end{theorem}
 \begin{proof}
 It suffices to prove that $f(t)=0$ for $t\leq q^{n_0}$. Since $f$ is almost automorphic, for
  the sequence of natural numbers $\mathbb{N}=\{n\}$, one can extract   a subsequence $\{n_{k}\}\subset \mathbb{N}$ such that

 \begin{equation}\label{p1}
   \lim\limits_{k\rightarrow\infty}f(tq^{n_{k}})=g(t), \,\,\,  \mathrm{for} \,\,\mathrm{ each} \,\, t\in \overline{q^\mathbb{Z}}\setminus\{0\}
\end{equation}
 and
 \begin{equation}\label{p2}
   \lim\limits_{k\rightarrow\infty}g(tq^{-n_{k}})=f(t), \,\,\,  \mathrm{for}\,\, \mathrm{each} \,\, t\in \overline{q^\mathbb{Z}}\setminus\{0\}.
 \end{equation}
It is clear that for any $t\leq q^{n_0}$, we can find $\{n_{kj}\}\subset\{n_{k}\}$ with $ tq^{n_{kj}}>q^{n_0}$ for all  $j=1,2,\ldots$. Thus, $f(tq^{n_{kj}})=0$ for all $j=1,2,\ldots$. By \eqref{p1}, $g(t)=\lim\limits_{j\rightarrow\infty}f(tq^{n_{kj}})=0$  for $ t\in \overline{q^\mathbb{Z}}\setminus\{0\}$. Hence, according to \eqref{p2}, we obtain $f(t)=0$ for $ t\in \overline{q^\mathbb{Z}}\setminus\{0\}$.
Since $f$ is continuous at $t=0$, $0=\lim\limits_{n\rightarrow -\infty}f(q^{n})=f(0).$ Therefore,  $f(t)=0$ for $ t\in \overline{q^\mathbb{Z}}$. The proof is complete.
\end{proof}

\begin{theorem}\label{t2.5}
Let $\{f_{n}\}$ be a sequence of almost automorphic functions such that $\lim\limits_{n\rightarrow\infty}f_{n}(t)=f(t)$ uniformly in $t\in \overline{q^\mathbb{Z}}$.
Then $f$ is almost automorphic.
\end{theorem}
\begin{proof}
For any given sequence $\{s'_n\}\subset \mathbb{Z}$, by the diagonal procedure one can extract a subsequence $\{s_{n}\}$ of $\{s'_n\}$ such that
\begin{equation}\label{ee21}
\lim\limits_{n\rightarrow\infty}f_{i}(tq^{s_{n}})=g_i(t)
\end{equation}
for each $i=1,2,...$ and each $t\in \overline{q^\mathbb{Z}}$.

We claim that the sequence of function $\{g_{i}(t)\}$ is a Cauchy sequence. In fact, for any $i,j\in \mathbb{N},$ we have
\begin{eqnarray*}
g_{i}(t)-g_{j}(t)&=&g_{i}(t)-f_{i}(tq^{s_{n}})+f_{i}(tq^{s_{n}})-f_{j}(tq^{s_{n}})\\
&&+f_{j}(tq^{s_{n}})-g_{j}(t),
\end{eqnarray*}
hence,
\begin{eqnarray*}
\|g_{i}(t)-g_{j}(t)\|&\leq&\|g_{i}(t)-f_{i}(tq^{s_{n}})\|+\|f_{i}(tq^{s_{n}})-f_{j}(tq^{s_{n}})\|\\
&&+\|f_{j}(tq^{s_{n}})-g_{j}(t)\|.
\end{eqnarray*}
For each $\varepsilon>0$, from the uniform convergence of $\{f_{n}\}$, there exists a positive integer $N(\varepsilon)$ such that for all $i,j>N$,
\[
\|f_{i}(tq^{s_{n}})-f_{j}(tq^{s_{n}})\|<\varepsilon,
\]
for all $t\in \overline{q^\mathbb{Z}}$, and all $n=1,2,\ldots$.

It follows from \eqref{ee21} and the completeness of the space $\mathbb{X}$ that the sequence $\{g_{i}(t)\}$ converges  pointwisely on $\overline{q^\mathbb{Z}}$ to a function, say to   function $g(t)$.

Now, we will prove
\[
\lim\limits_{n\rightarrow\infty}f(tq^{s_{n}})=g(t)
\]
and
\[
\lim\limits_{n\rightarrow\infty}g(tq^{-s_{n}})=f(t)
\]
pointwise on $\overline{q^\mathbb{Z}}$.

Indeed for each $i=1,2,\ldots$, we have
\begin{eqnarray}\label{p3}
\|f(tq^{s_{n}})-g(t)\|&\leq&\|f(tq^{s_{n}})-f_{i}(tq^{s_{n}})\|\nonumber\\
&&+\|f_{i}(tq^{s_{n}})-g_{i}(t)\|+\|g_{i}(t)-g(t)\|.
\end{eqnarray}
For any $\varepsilon>0$, we can find some positive integer $N_0(t,\varepsilon)$ such that
\[
\|f(tq^{s_{n}})-f_{N_0}(tq^{s_{n}})\|<\varepsilon
\]
for every $t\in \overline{q^\mathbb{Z}}, n=1,2,\ldots$, and $\|g_{N_0}(t)-g(t)\|<\varepsilon$ for every $t\in \overline{q^\mathbb{Z}}$. Hence, by \eqref{p3}, we get
\begin{equation}\label{p4}
  \|f(tq^{s_{n}})-g(t)\|<2\varepsilon+\|f_{N_0}(tq^{s_{n}})-g_{N_0}(t)\|
\end{equation}
for every $t\in \overline{q^\mathbb{Z}}, n=1,2,\ldots$.

In view of \eqref{ee21}, for every $t\in \overline{q^\mathbb{Z}}$, there is some positive integer $M=M(t,N_0)$ such that
\[
\|f_{N_0}(tq^{s_{n}})-g_{N_0}(t)\|<\varepsilon
\]
for every $n>M$. From this and \eqref{p4}, we obtain
\[
\|f(tq^{s_{n}})-g(t)\|<3\varepsilon
\]
for $n\geq N_{0}(t,\varepsilon)$.

Similarly, we can prove that
\[
\lim\limits_{n\rightarrow\infty}g(tq^{s_{n}})=f(t)\,\, \mathrm{for} \,\ \mathrm{each} \,\ t\in \overline{q^\mathbb{Z}}.
\]
The proof is complete.
\end{proof}

\begin{remark}If we
denote by $AA(\mathbb{X})$, the set of all  almost automorphic functions $f:\overline{q^\mathbb{Z}}\rightarrow \mathbb{X}$, then by Theorem \ref{t21}, we see that $AA(\mathbb{X})$ is a vector space, and
according to Theorem \ref{t2.5}, this vector space
equipped with the norm
$$\|f\|_{AA(\mathbb{X})}=\sup\limits_{t\in\overline{ q^\mathbb{Z}}}\|f(t)\|$$
is a Banach space.
\end{remark}

\begin{definition}\label{d23}
A continuous function $f:\mathbb{R}\times \mathbb{X} \rightarrow \mathbb{X}$ is said to be almost automorphic in  $t \in \overline{q^\mathbb{Z}}$ for each $x \in \mathbb{X}$,
if for each sequence of integer numbers $\{s_{n}^{'}\}$, there exists a subsequence $\{s_{n}\}$ such that
\[
\lim\limits_{n\rightarrow\infty}f(tq^{s_{n}},x)=g(t,x)\]
 exists for each $t \in \overline{q^\mathbb{Z}}$ and each  $x \in \mathbb{X}$, and
 \[
\lim\limits_{n\rightarrow\infty}g(tq^{-s_{n}},x)=f(t,x)
\]
 exists for each $t \in \overline{q^\mathbb{Z}}$ and each  $x \in \mathbb{X}$.
\end{definition}

\begin{theorem}\label{t26}
If $f_{1},f_{2}: \overline{q^\mathbb{Z}}\times \mathbb{X} \rightarrow \mathbb{X}$ are almost automorphic functions  in $t$ for each  $x \in \mathbb{X}$, then the following functions are also  almost automorphic in $t$ for each  $x \in \mathbb{X}$:
\begin{itemize}
  \item [$(i)$] $f_{1}+f_{2}$,
  \item [$(ii)$] $c f_{1}$, $c$ is an arbitrary scalar.
\end{itemize}
\end{theorem}
\begin{proof}  The proof is obvious. We omit it here. The proof is complete.
\end{proof}

\begin{theorem}\label{t27}
If $f(t,x) $ are almost automorphic in $t$ for each $x\in \mathbb{X}$, then
$$\sup\limits_{t\in \overline{q^\mathbb{Z}}}\|f(t,x)\|=M_{x}<\infty $$
for each $x\in \mathbb{X}$.
\end{theorem}
\begin{proof} Suppose not. Assume, to the contrary, that
$$\sup\limits_{t\in \overline{q^\mathbb{Z}}}\|f(t,x_{0})\|=\infty $$
for some $x_{0}\in \mathbb{X}$. Thus, there exists a  sequence of integer numbers $\{s_{n}^{'}\}$ such that
$$\lim\limits_{n\rightarrow\infty}\|f(q^{s_{n}^{'}},x_{0})\|=\infty. $$
Since $f(t,x_{0})$ is almost automorphic in $t$, one can extract a subsequence $\{s_{n}\}$ from $\{s_{n}^{'}\}$ such that
$$\sup\limits_{t\in \overline{q^\mathbb{Z}}}\|f(q^{s_{n}},x_{0})\|=g(1,x_{0}),$$
which is a contradiction. The proof is complete.
\end{proof}

\begin{theorem}\label{t2.2.4}
If $f $ is  almost  automorphic   in $t$ for each  $x \in \mathbb{X}$, then the function $g$ in Definition \ref{d23} satisfies
$$\sup\limits_{t\in\mathbb{R}}\|g(t,x)\|=N_{x}<\infty$$
for each $x\in \mathbb{X}$.
\end{theorem}
\begin{proof}
The proof is obvious. We omit it here. The proof is complete.
\end{proof}

\begin{theorem}\label{t28}
If $f$ is almost automorphic in $t$ for each  $x \in \mathbb{X}$ and if $f$ is Lipschitzian in  $x$ uniformly in $t$, that is, there exists a positive constant $L>0$ such that for each pair $x,y\in \mathbb{X}$,
$$\|f(t,x)-f(t,y)\|<L\|x-y\| $$
uniformly in $t\in \overline{q^\mathbb{Z}}$, then $g$ satisfies the same Lipschitz condition in $x$ uniformly in  $t$.
\end{theorem}
\begin{proof}Because for each sequence of integer numbers $\{s_{n}^{'}\}$, there exists a subsequence $\{s_{n}\}$ such that
\[
\lim\limits_{n\rightarrow\infty}f(tq^{s_{n}},x)=g(t,x)\]
 exists for each $t \in \overline{q^\mathbb{Z}}$ and each  $x \in \mathbb{X}$,  so  for any $t\in \overline{q^\mathbb{Z}}$  and any given $\varepsilon> 0$, we have
$$\|g(t,x)-f(tq^{s_{n}},x)\|<\frac{\varepsilon}{2}$$ and
$$\|g(t,y)-f(tq^{s_{n}},y)\|<\frac{\varepsilon}{2}$$
for $n$ sufficiently large.

Hence, for $n$ sufficiently large we find
\begin{eqnarray*}
||g(t,x)-g(t,y)||&=&||g(t,x)-f(tq^{s_{n}},x)+f(tq^{s_{n}},x)-f(tq^{s_{n}},y)\\
&&+f(tq^{s_{n}},y)-g(t,y)||\\
&<&\varepsilon+L\|x-y\|.
\end{eqnarray*}
Letting $\varepsilon\rightarrow 0^+$, we get
$$\|g(t,x)-g(t,y)\|\leq L\|x-y\|$$
for each $x,y\in \mathbb{X}$.
The proof is complete.
\end{proof}

\begin{theorem}\label{t29}
 Let  $f: \overline{q^\mathbb{Z}}\times \mathbb{X}\rightarrow \mathbb{X}$ be almost automorphic  in $t$ for each $x\in \mathbb{X}$ and assume that $f$  satisfies a Lipschitz in $x$ uniformly in  $t\in \overline{q^\mathbb{Z}}$. Let $\varphi: \overline{q^\mathbb{Z}}\rightarrow \mathbb{X}$ be almost automorphyic. Then the function  $F: \overline{q^\mathbb{Z}}\rightarrow \mathbb{X}$ defined by $F(t)=f(t,\varphi(t))$ is almost automorphic.
\end{theorem}
\begin{proof}It is easy to see that for any given sequence $\{s'_n\}$, there exists
  a subsequence  $\{s_{n}\}\subset\{s'_n\}$ such that
\begin{equation}\label{p6}
  \lim\limits_{n\rightarrow\infty} f(tq^{s_{n}},x)=g(t,x)
\end{equation}
for each $t\in \overline{q^\mathbb{Z}}$ and $x\in \mathbb{X}$,
\begin{equation}\label{p7}
\lim\limits_{n\rightarrow\infty} \varphi(tq^{s_{n}})=\phi(t)
\end{equation}
 for each $t\in \overline{q^\mathbb{Z}}$,
 $$\lim\limits_{n\rightarrow\infty} g(tq^{-s_{n}},x)=f(t,x)$$ for each $t\in \overline{q^\mathbb{Z}}$ and $x\in \mathbb{X}$, and
  $$\lim\limits_{n\rightarrow\infty} \phi(tq^{-s_{n}})=\varphi(t)$$ for each $t\in \overline{q^\mathbb{Z}}$.

Consider the function $G: \overline{q^\mathbb{Z}}\rightarrow \mathbb{X}$  defined by $G(t)=g(t,\phi(t))$, $t\in \overline{q^\mathbb{Z}}$. We will show that $\lim\limits_{n\rightarrow\infty}F(tq^{s_{n}})=G(t)$, for each $t\in \overline{q^\mathbb{Z}}$ and  $\lim\limits_{n\rightarrow\infty}G(tq^{-s_{n}})=F(t)$,  for each $t\in  \overline{q^\mathbb{Z}}$.

In fact, noting that
\begin{eqnarray*}
||F(tq^{s_{n}})-G(t)||&=&||f(tq^{s_{n}},\varphi(tq^{s_{n}}))-f(tq^{s_{n}},\phi(t))\\
&&+f(tq^{s_{n}},\phi(t))-g(t,\phi(t))||\\
&\leq&L\|\varphi(tq^{s_{n}})-\phi(t)\|\\
&&+\|f(tq^{s_{n}},\phi(t))-g(t,\phi(t))\|,
\end{eqnarray*}
by \eqref{p6} and \eqref{p7}, we  get
$$\lim\limits_{n\rightarrow\infty} F(tq^{s_{n}})=G(t), \,\,\mathrm{ for} \,\, \mathrm{each} \,\ t\in \overline{q^\mathbb{Z}}.$$
Similarly we can prove that $\lim\limits_{n\rightarrow\infty}G(tq^{-s_{n}})=F(t)$
for each $t\in\overline{ q^\mathbb{Z}}$. This completes the proof.
\end{proof}

Before ending this section, we give the second type of concepts of almost automorphic functions on the quantum time scale  as follows:

\begin{definition}\label{d2.1b}
Let $\mathbb{X}$ be a $($real or complex$)$ Banach space and $f:\overline{q^\mathbb{Z}}\rightarrow \mathbb{X}$ a $($strongly$)$ continuous function. We say that $f$ is almost automorphic if for every sequence of integer numbers $\{s'_{n}\}\subset \mathbb{Z}$, there exists a subsequence $\{s_{n}\}$ such that:
\[
g(t):=\lim\limits_{n\rightarrow\infty}q^{s_{n}}f(tq^{s_{n}})
\]
is well defined for each $t\in \overline{q^\mathbb{Z}}$ and
\[
\lim\limits_{n\rightarrow\infty}q^{-s_{n}}g(tq^{-s_{n}})=f(t)
\]
for each $t\in \overline{q^\mathbb{Z}}$.
\end{definition}

\begin{definition}\label{d23b}
A continuous function $f:\mathbb{R}\times \mathbb{X} \rightarrow \mathbb{X}$ is said to be almost automorphic in  $t \in \overline{q^\mathbb{Z}}$ for each $x \in \mathbb{X}$,
if for each sequence of integer numbers $\{s_{n}^{'}\}$, there exists a subsequence $\{s_{n}\}$ such that
\[
\lim\limits_{n\rightarrow\infty}q^{s_{n}}f(tq^{s_{n}},x)=g(t,x)\]
 exists for each $t \in \overline{q^\mathbb{Z}}$ and each  $x \in \mathbb{X}$, and
 \[
\lim\limits_{n\rightarrow\infty}q^{-s_{n}}g(tq^{-s_{n}},x)=f(t,x)
\]
 exists for each $t \in \overline{q^\mathbb{Z}}$ and each  $x \in \mathbb{X}$.
\end{definition}

\begin{remark}It is easy to check that
all the results of this section hold for almost automorphic functions defined by Definitions \ref{d2.1} and \ref{d23} are also valid for  almost automorphic functions defined by Definitions  \ref{d2.1b} and \ref{d23b}.
\end{remark}

\section{An equivalent definition of almost automorphic functions on the quantum time scale}
\setcounter{equation}{0}

\indent

In this section, we will give an equivalent definition of almost automorphic functions on the quantum time scale $\overline{q^\mathbb{Z}}$.
To this end,
we introduce a notation $-\infty_q$ and  stipulate $q^{-\infty_q}=0, t\pm(-\infty_q)=t$ and $t>-\infty_q$ for all $t\in \mathbb{Z}$.
Let $f\in C(\overline{q^\mathbb{Z}},\mathbb{X})$, we define a function $\tilde{f}:  \mathbb{Z}\cup \{-\infty_q\}\rightarrow \mathbb{X}$ by
\begin{eqnarray}\label{g1}
\tilde{f}(t)=\left\{\begin{array}{lll}
f(q^t),& t\in \mathbb{Z},\\
f(0),& t=-\infty_q,\end{array}\right.
\end{eqnarray}
that is,
\begin{eqnarray*}
f(t)=\left\{\begin{array}{lll}
\tilde{f}(\log_q t),& t\in q^\mathbb{Z},\\
\lim\limits_{t\rightarrow 0^+}f(t),& t=0.\end{array}\right.
\end{eqnarray*}
 Since $f(t)$ is right continuous at $t=0$, it is clear that the above definition is well defined.

Moreover, for $f\in C(\overline{q^\mathbb{Z}}\times \mathbb{X},\mathbb{X})$, we define a function $\tilde{f}:  \mathbb{Z}\cup \{-\infty_q\}\times \mathbb{X}\rightarrow \mathbb{X}$ by
\begin{eqnarray}\label{g2}
\tilde{f}(t,x)=\left\{\begin{array}{lll}
f(q^t,x),& (t,x)\in \mathbb{Z}\times \mathbb{X},\\
f(0,x),& t=-\infty_q, x\in \mathbb{X},\end{array}\right.
\end{eqnarray}
that is,
\begin{eqnarray*}
f(t,x)=\left\{\begin{array}{lll}
\tilde{f}(\log_q t,x),& (t,x)\in q^\mathbb{Z}\times \mathbb{X},\\
\lim\limits_{t\rightarrow 0}f(t,x),& t=0, x\in \mathbb{X}.\end{array}\right.
\end{eqnarray*}
Since $f(t,x)$ is continuous at $(0,x)$, it is clear that the above definition is well defined.

\begin{definition}\label{feq1}
A function $ f: \mathbb{Z}\cup\{-\infty_q\}\rightarrow\mathbb{X}$ is called almost automorphic
if for every sequence $(s_n')\subset \mathbb{Z}$
 there exists a subsequence $(s_n)\subset (s_n')$
such that
\[
\lim\limits_{n\rightarrow\infty} f(t+s_n)= g(t)
\]
is well defined for each $t \in  \mathbb{Z}\cup\{-\infty_q\}$, and
\[
\lim\limits_{n\rightarrow \infty} g(t-s_n) = f(t)
\]
for each $t \in \mathbb{Z}\cup\{-\infty_q\}$.
\end{definition}

\begin{definition}\label{feq2}
A  function $F: (\mathbb{Z}\cup\{-\infty_q\})\times \mathbb{X}\rightarrow \mathbb{X}$ is
called almost automorphic if for every sequence
 $(s_n')\subset \mathbb{Z}$
 there exists a subsequence $(s_n)\subset \mathbb{Z}$
such that
\[
\lim\limits_{n\rightarrow\infty} F(t+s_n,x)= G(t,x)
\]
is well defined for each $t \in  \mathbb{Z}\cup\{-\infty_q\}$, and
\[
\lim\limits_{n\rightarrow \infty} G(t-s_n,x) = F(t,x)
\]
for each $t \in \mathbb{Z}\cup\{-\infty_q\}$ and $x\in \mathbb{X}$.
\end{definition}

\begin{remark}\label{re31}We can view $ \mathbb{Z}\cup\{-\infty_q\}$ as a kind of  generalized integer number set. 
Obviously, the automorphic functions defined by Definitions \ref{feq1} and \ref{feq2}(which are defined on $\mathbb{Z}\cup\{-\infty_q\}$ or $\mathbb{Z}\cup\{-\infty_q\}\times \mathbb{X}$) share the same properties as  the ordinary automorphic functions  defined on $\mathbb{Z}$ or $\mathbb{Z}\times \mathbb{X}$.
\end{remark}

\begin{definition}\label{bc1}
A  function   $f\in C(\overline{q^\mathbb{Z}},\mathbb{X})$ is
called almost automorphic if and only if  the function $\tilde{f}(t)$ defined by \eqref{g1} is almost automorphic.
\end{definition}

\begin{definition}\label{bc2}
A  function   $f\in C(\overline{q^\mathbb{Z}}\times \mathbb{X},\mathbb{X})$ is
called almost automorphic in  $t \in \overline{q^\mathbb{Z}}$ for each $x \in \mathbb{X}$ if and only if  the function $\tilde{f}(t,x)$ defined by \eqref{g2} is almost automorphic in  $t \in \overline{q^\mathbb{Z}}$ for each $x \in \mathbb{X}$.
\end{definition}

Obviously, Definitions \ref{bc1} and \ref{bc2} are equivalent to Definitions \ref{d2.1} and \ref{d23}, respectively.
Moreover, by Remark \ref{re31}, all of the properties of almost automorphic functions on the quantum time scale can be directly obtained from the corresponding properties of the ordinary almost automorphic functions defined on $\mathbb{Z}$ or $\mathbb{Z}\times \mathbb{X}$.

\section{Automorphic solutions for semilinear dynamic equations on the quantum time scale}
\setcounter{equation}{0}

\indent

In this section,  we will study the existence of automorphic solutions of semilinear dynamic equations on the quantum time scale. Throughout this section, we use the letter $\mathbb{E}$ to stand for either $\mathbb{R}$ or $\mathbb{C}$.

Consider the semilinear dynamic equation on the  quantum time scale:
\begin{equation}\label{qqn1}
D_qx  (t)=B(t)x(t)+g(t,x(t),x(tq^{-\sigma(t)}))),\,\, t\in \overline{q^\mathbb{Z}},
\end{equation}
where $\sigma: \mathbb{T} \rightarrow [0,\infty)_\mathbb{T}$ is a scalar delay function and satisfies $t-\sigma(t)\in \mathbb{T}$ for all $t\in \mathbb{T}$, $B(t)$ is a regressive, rd-continuous $n\times n$ matrix
valued function, $g\in C_{rd}(\mathbb{T}\times \mathbb{E}^{2n},\mathbb{E}^n)$.
 Under the transformation \eqref{g2}, equation \eqref{qqn1} is transformed to
\begin{equation}\label{qqn2}
\Delta \tilde{x}  (n)=A(n)\tilde{x}(n)+f(n,\tilde{x}(n),\tilde{x}(n-\tau(n))),\,\,n\in \mathbb{Z}\cup\{-\infty_q\},
\end{equation}
and vice visa, where $A(n)=(q-1)q^n\tilde{B}(n), f(n)=(q-1)q^n\tilde{g}(n,\tilde{x}(n),\tilde{x}(n-\overline{\sigma}(n)))),\tau(n)=\tilde{\sigma}(n).$

Clearly, if $x(t)$ is a solution of \eqref{qqn1} if and only if  $\tilde{x}(n)$ is a solution of \eqref{qqn2}.

\begin{definition}\cite{ra}
Let  $A(t)$ be an $n\times n$
rd-continuous matrix value function on $\mathbb{T}$, the linear system
\begin{eqnarray}\label{e21}
x^{\Delta}(t)=A(t)x(t),\,\, t\in\mathbb{T}
\end{eqnarray}
is said to admit an exponential dichotomy on $\mathbb{T}$ if there
exist positive constants $K_1,K_2$ and $\alpha_1,\alpha_2$, an invertible projection $P$ commuting with $X(t)$, where $X(t)$ is principal
fundamental matrix solution  of \eqref{e21} satisfying
\begin{eqnarray*}
&&||X(t)PX^{-1}(s)||\leq K_1e_{\ominus
\alpha_1}(t,s),\,\,
s, t \in\mathbb{T}, t \geq s,\\
&&||X(t)(I-P)X^{-1}(s)||\leq K_2e_{\ominus
\alpha_2}(s,t),\,\, s, t \in\mathbb{T}, t \leq s.
\end{eqnarray*}
\end{definition}
\begin{theorem} \cite{ra} Let $\mathbb{T}$ be an almost periodic time scale.
Suppose that
the linear homogeneous system \eqref{e21} admits an exponential dichotomy with the
positive constants $K_1,K_2$ and $\alpha_1,\alpha_2$ and invertible projection $\mathcal{P}$ commuting with
$X(t)$, where $X(t)$ is principal fundamental matrix solution of \eqref{e21},
 then the nonhomogeneous system
\begin{equation}
x^\Delta (t)=A(t)x(t)+f(t),\label{eq1}
\end{equation}
has a solution $x(t)$ of the form
\begin{equation}\label{eq2}
x(t)=\int_{-\infty}^{t}X(t)\mathcal{P}X^{-1}(\sigma(s))f(s)\Delta s-\int_{t}^{\infty}X(t)(1-\mathcal{P})X^{-1}(\sigma(s))f(s)\Delta s.
\end{equation}
 Moreover, we have
\begin{equation*}
\|x\|\leq \bigg(\frac{K_1+\alpha_1}{\alpha_1}+\frac{K_2}{\alpha_2}\bigg)\|f\|.
\end{equation*}
\end{theorem}

Consider the following semilinear dynamic equation on almost periodic time scale $\mathbb{T}$:
\begin{equation}\label{qqqn1}
x^\Delta (t)=A(t)x(t)+f(t,x(t),x(t-\tau(t)))),
\end{equation}
where $\tau: \mathbb{T} \rightarrow [0,\infty)_\mathbb{T}$ is a scalar delay function and satisfies $t-\tau(t)\in \mathbb{T}$ for all $t\in \mathbb{T}$, $A(t)$ is a regressive, rd-continuous $n\times n$ matrix
valued function, $f\in C_{rd}(\mathbb{T}\times \mathbb{E}^{2n},\mathbb{E}^n)$. The corresponding linear homogeneous system of \eqref{qqqn1} is
\begin{equation}\label{qqqn2}
x^\Delta (t)=A(t)x(t).
\end{equation}

We make the following assumptions:
\begin{itemize}
  \item [$(A_1)$] Functions $\tau(t), A(t)$ and $f(t, u, v)$ are almost automorphic in $t$.
  \item [$(A_2)$]There exists a constant $L_1,L_2> 0$ such that
  \[
  ||f(t,u_1,v_2)-f(t,u_2,v_2)||\leq L_1||u_1-u_2||+L_2||v_1-v_2||
  \]
  for all $t\in \mathbb{T}$ and for any vector valued functions $u$ and $v$ defined on $\mathbb{T}$.
  \item [$(A_3)$]The linear homogeneous system \eqref{qqqn2} admits an exponential dichotomy with the
positive constants $K_1,K_2$ and $\alpha_1,\alpha_2$ and invertible projection $P$ commuting with
$X(t)$, where $X(t)$ is principal fundamental matrix solution of \eqref{qqqn2}.
\end{itemize}

Now, define the mapping $\Psi$ by
\begin{eqnarray*}
(\Psi x)(t)&:=& \int_{-\infty}^{t}X(t)\mathcal{P}X^{-1}(\sigma(s))f(s,x(s),x(s-\tau(s)))\Delta s\nonumber\\
&&-\int_{t}^{\infty}X(t)(1-\mathcal{P})X^{-1}(\sigma(s))f(s,x(s),x(s-\tau(s)))\Delta s.
\end{eqnarray*}

The following result can be proven similar to Lemma 6 in \cite{ra}, hence we omit
it.
\begin{lemma}Suppose $(A_1)$-$(A_3)$ hold. Then
the mapping $\Psi$ maps $\mathbb{A}\mathbb{A}(\mathbb{E}^n)$ into $\mathbb{A}\mathbb{A}(\mathbb{E}^n)$.
\end{lemma}

\begin{theorem}\label{thmzy1}
Suppose $(A_1)$-$(A_3)$ hold. Assume further that
\begin{itemize}
  \item [$(A_4)$]$\big(\frac{K_1+\alpha_1}{\alpha_1}+\frac{K_2}{\alpha_2}\big)(L_1+L_2)<1$.
\end{itemize}
Then
\eqref{qqqn1} has a unique almost automorphic solution.
\end{theorem}
\begin{proof}
For any $x,y\in \mathbb{A}\mathbb{A}(\mathbb{E}^n)$, we have
\begin{eqnarray*}
&&||\Psi x-\Psi y||\nonumber\\
&=& \sup\limits_{t\in \mathbb{T}}\bigg|\int_{-\infty}^{t}X(t)\mathcal{P}X^{-1}(\sigma(s))[f(s,x(s),x(s-\tau(s)))-f(s,y(s),y(s-\tau(s)))]\Delta s\nonumber\\
&&-\int_{t}^{\infty}X(t)(I-\mathcal{P})X^{-1}(\sigma(s))[f(s,x(s),x(s-\tau(s)))-f(s,y(s),y(s-\tau(s)))]\Delta s\bigg|\nonumber\\
&\leq& \sup\limits_{t\in \mathbb{T}}\bigg|\int_{-\infty}^{t}K_1e_{\ominus\alpha_1}(t,\sigma(s))(L_1+L_2)||x-y||\Delta s\nonumber\\
&&-\int_{t}^{\infty}K_2e_{\ominus\alpha_2}(\sigma(s),t)(L_1+L_2)||x-y||\Delta s\bigg|
\nonumber\\
&\leq& \bigg(\frac{K_1+\alpha_1}{\alpha_1}+\frac{K_2}{\alpha_2}\bigg)(L_1+L_2)||x-y||
\end{eqnarray*}
Hence,  $\Phi$ is a
contraction.  Therefore,  $\Phi$ has a unique fixed point in  $\mathbb{A}\mathbb{A}(\mathbb{E}^n)$, so, \eqref{qqqn1} has a unique
almost automorphic solution.
\end{proof}

In Theorem \ref{thmzy1}, if we take $\mathbb{T}=\mathbb{Z}\cup\{-\infty_q\}$, then we have
\begin{theorem}
Suppose $(A_1)$-$(A_4)$ hold. Then
\eqref{qqn2} has a unique almost automorphic solution, and so \eqref{qqn1} has a unique almost automorphic solution.
\end{theorem}

Consider a linear quantum difference equation
\begin{equation}
D_qx(t)=A(t)x(t)+f(t),\,\, t \in \overline{q^\mathbb{Z}}, \label{qq1}
\end{equation}
where $A$ is an $n\times n$ matrix
valued function and $f$ is an $n$-dimensional vector valued function.
Under the transformation \eqref{g1}, \eqref{qq1} transforms to
\begin{equation}\label{qq2}
\Delta \tilde{x} (n)=(q-1)q^n\tilde{A}(n)\tilde{x}(n)+(q-1)q^n\tilde{f}(n),\,\, n \in \mathbb{Z}\cup\{-\infty_q\},
\end{equation}
and vice versa.

Consider the following non-autonomous linear difference equation
\begin{equation}\label{az1}
  x(k+1)=A(k)x(k)+f(k),\,\, k\in \mathbb{Z}\cup\{-\infty_q\},
\end{equation}
where $A(k)$ are given non-singular $n\times n$ matrices with elements $a_{ij}(k), 1 \leq i, j \leq n, f : \mathbb{Z}\rightarrow\mathbb{E}^n$ is a given $n\times 1$ vector
function and $x(k)$ is an unknown $n\times 1$ vector with components $x_i(k), 1 \leq i \leq n$. Its associated homogeneous equation is given by
\begin{equation}\label{az2}
  x(k+1)=A(k)x(k),\,\, k\in \mathbb{Z}\cup\{-\infty_q\}.
\end{equation}

Similar to Definition 2.11 in \cite{azz}, we give the following definition:
\begin{definition}  Let $U(k)$ be the principal fundamental matrix of the difference system \eqref{az2}. The system  \eqref{az2} is said to
possess an exponential dichotomy if there exists a projection $P$, which commutes with $U(k)$, and positive constants $\eta, \nu, \alpha, \beta$
such that for all $k, l \in \mathbb{Z}\cup\{-\infty_q\}$, we have
\begin{eqnarray*}
& || U(k)PU^{-1}(l)|| \leq \eta e^{-\alpha(k-l)},\,\,k\geq l,\\
 &||U(k)(I-P)U^{-1}(l)||\leq \upsilon e^{-\beta(l-k)},\,\, l\geq k.
\end{eqnarray*}
\end{definition}

Similar to the proof of Theorem 3.1 in \cite{azz}, one can easily show that
\begin{theorem} Suppose $A(k)$ is discrete almost automorphic and a non-singular matrix and the set $\{A^{-1}(k)\}_{k\in \mathbb{Z}\cup\{-\infty_q\}}$ is bounded. Also,
suppose the function $f : \mathbb{Z}\cup\{-\infty_q\}\rightarrow \mathbb{E}^n$ is a discrete almost automorphic function and Eq. \eqref{az2} admits an exponential dichotomy with
positive constants $\nu, \eta, \beta$ and $\alpha$. Then, the system \eqref{az1} has an almost automorphic solution on $\mathbb{Z}\cup\{-\infty_q\}$.
\end{theorem}

\begin{corollary}
Suppose $B(n):=(q-1)q^n\tilde{A}(n)+I$ is discrete almost automorphic and a non-singular matrix and the set $\{B^{-1}(n)\}_{n\in \mathbb{Z}\cup\{-\infty_q\}}$ is bounded. Also,
suppose the function $g:=(q-1)q^n\tilde{f}(n) : \mathbb{Z}\cup\{-\infty_q\}\rightarrow \mathbb{E}^n$ is a discrete almost automorphic function and equation
\[
\Delta y(n)=B(n)y(n)+g(n)
\] admits an exponential dichotomy with
positive constants $\nu, \eta, \beta$ and $\alpha$. Then, the system \eqref{qq1} has an almost automorphic solution on $\overline{q^\mathbb{Z}}$.
\end{corollary}

\section{ Conclusion}
\setcounter{equation}{0}

\indent

In this paper, we proposed two types of concepts of almost automorphic functions on the quantum time scale and studied some of their basic properties. Moreover, based on the transformation between functions defined on the quantum time scale and functions defined on the set of generalized integer numbers, we gave equivalent definitions of   almost automorphic functions on the quantum time scale.
As an application of our results, we established the existence of almost automorphic solutions for semilinear dynamic equations on the quantum time scale.
By using the methods and results of this paper,
for example, one can study the almost automorphy of neural networks on the quantum time scale and population dynamical models on the quantum time scale and so on. Furthermore,
by using the transformation and the set of generalized integer numbers introduced in Section 3 of this paper,  one can propose   concepts of  almost periodic functions,
  pseudo almost periodic functions,
   weighted pseudo
almost automorphic functions,  almost periodic set-valued functions, almost periodic functions in the sense
of Stepanov  on the quantum time scale and so on.

\end{document}